\documentclass[
final
]{dmtcs-episciences}

\usepackage[utf8]{inputenc}
\usepackage{subfigure}

\usepackage[round]{natbib}

\usepackage{amsthm}
\usepackage{amsmath}
\usepackage{amsfonts}
\usepackage{amssymb}
\usepackage{epic}
\usepackage{epstopdf}
\usepackage{tabularx}
\usepackage{array}
\usepackage{multirow}
\usepackage{longtable}
\usepackage{tikz}
\usepackage{hyperref}
\usepackage{array}
\newcolumntype{P}[1]{>{\raggedright\arraybackslash}p{#1}}
\newcolumntype{L}[1]{>{\raggedright\let\newline\\\arraybackslash\hspace{0pt}}m{#1}}
\newcolumntype{C}[1]{>{\centering\let\newline\\\arraybackslash\hspace{0pt}}m{#1}}
\newcolumntype{R}[1]{>{\raggedleft\let\newline\\\arraybackslash\hspace{0pt}}m{#1}}

\newtheorem{theorem}{Theorem}
\newtheorem{lemma}{Lemma}

\newtheorem{conjecture}{Conjecture}

\newtheorem{proposition}{Proposition}

\newcommand{\ignore}[1]{}

\author{Naiomi T. Cameron\affiliationmark{1}
  \and Kendra Killpatrick\affiliationmark{2}
}
\title{Statistics on Linear Chord Diagrams}
\affiliation{
  Spelman College, Atlanta, GA, USA\\
  Pepperdine University, Malibu, CA, USA}
\keywords{linear chord diagrams, matchings, noncrossing partitions, Narayana numbers, Eulerian triangle}
\received{2019-2-26}
\revised{2020-1-5}
\accepted{2020-1-7}
\begin{document}
\publicationdetails{21}{2019}{2}{11}{5211}
\maketitle
\begin{abstract}
  Linear chord diagrams are partitions of $\left[2n\right]$ into $n$ blocks of size two called chords. We refer to a block of the form $\{i,i+1\}$ as a short chord. In this paper, we study the distribution of the number of short chords on the set of linear chord diagrams, as a generalization of the Narayana distribution obtained when restricted to the set of noncrossing linear chord diagrams. We provide a combinatorial proof that this distribution is unimodal and has an expected value of one. We also study the number of pairs $(i,i+1)$ where $i$ is the minimal element of a chord and $i+1$ is the maximal element of a chord. We show that the distribution of this statistic on linear chord diagrams corresponds to the second-order Eulerian triangle and is log-concave.

\end{abstract}

\section{Introduction}\label{background}

Linear chord diagrams are partitions of $\left[2n\right]=\{1,2,...,2n\}$ into $n$ blocks of size two called {\em chords}. For a given chord $\{ a, b \}$  with $b>a$, we call $a$ the {\em startpoint}, $b$ the {\em endpoint} and $b-a$ the {\em length} of the chord.  Linear chord diagrams with $n$ chords are also called (perfect) matchings on $[2n]$, i.e. ways of connecting $2n$ points in the plane lying on a horizontal line by $n$ arcs, each arc connecting two of the points and lying above the points, as described by \cite{StaAdd}.

Two chords, $\{a_1, b_1\}$ and $\{a_2, b_2 \}$, are said to be {\em crossing} if $a_1 < a_2<b_1 < b_2$ and are said to be {\em nesting} if $a_1 < a_2<b_2 < b_1$. For example, when $n=2$ there are three linear chord diagrams: 

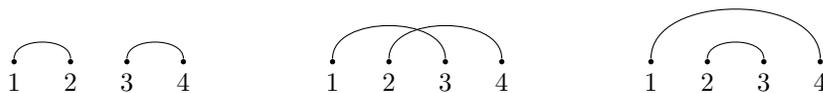
\begin{tabular}{p{1.5in}p{1.5in}p{1.5in}}

\begin{tikzpicture}[scale=0.75]
 \tikzstyle{black} = [circle, minimum width=2pt, fill, inner sep=0pt]
 \node[black] (n1) [label=below:$1$] at (0,0) {};
 \node[black] (n2) [label=below:$2$] at (1,0) {};
 \node[black] (n3) [label=below:$3$] at (2,0) {};
  \node[black] (n4) [label=below:$4$] at (3,0) {};
\draw (n1) to [out=90,in=90] (n2);
\draw (n3) to [out=90,in=90] (n4);
\end{tikzpicture}

&

\begin{tikzpicture}[scale=0.75]
 \tikzstyle{black} = [circle, minimum width=2pt, fill, inner sep=0pt]
 \node[black] (n1) [label=below:$1$] at (0,0) {};
 \node[black] (n2) [label=below:$2$] at (1,0) {};
 \node[black] (n3) [label=below:$3$] at (2,0) {};
  \node[black] (n4) [label=below:$4$] at (3,0) {};
  \draw (n1) to [out=90,in=90] (n3);
\draw (n2) to [out=90,in=90] (n4);
\end{tikzpicture}

&

\begin{tikzpicture}[scale=0.75]
 \tikzstyle{black} = [circle, minimum width=2pt, fill, inner sep=0pt]
 \node[black] (n1) [label=below:$1$] at (0,0) {};
 \node[black] (n2) [label=below:$2$] at (1,0) {};
 \node[black] (n3) [label=below:$3$] at (2,0) {};
  \node[black] (n4) [label=below:$4$] at (3,0) {};
  \draw (n1) to [out=90,in=90] (n4);
\draw (n2) to [out=90,in=90] (n3);
\end{tikzpicture}
\end{tabular}

\noindent The first of the three linear chord diagrams above consists of a pair of chords that are neither crossing nor nesting; the second diagram consists of exactly one crossing and the third diagram consists of exactly one nesting. 

When $n=3$, there are fifteen linear chord diagrams, all of which are shown in Figure \ref{linchordn3} below. The diagrams in the first row of Figure \ref{linchordn3} are noncrossing (i.e., contain no crossings), while the last diagram in the first row and the first four diagrams in the second row of Figure \ref{linchordn3} are nonnesting (i.e., contain no nestings). In general, the total number of linear chord diagrams having $n$ chords is $(2n-1)!!$. 


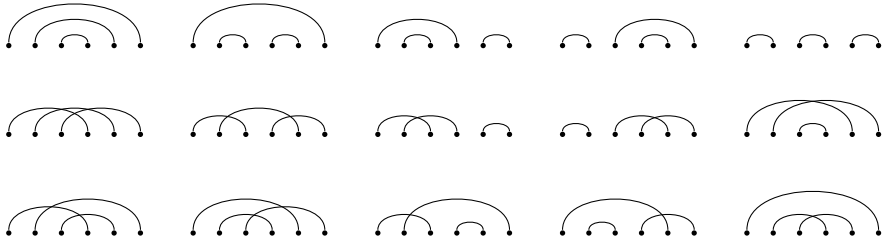
\begin{figure}[htbp]
\begin{center}
\begin{tabular}{p{0.8in}p{0.8in}p{0.8in}p{0.8in}p{0.8in}}

\begin{tikzpicture}[scale=0.35]
\tikzstyle{black} = [circle, minimum width=2pt, fill, inner sep=0pt]
\foreach \x in {1,...,6}
\node[black] (n\x) at (\x,0) {};
\draw (n1) to [out=90,in=90] (n6);
\draw (n2) to [out=90,in=90] (n5);
\draw (n3) to [out=90,in=90] (n4);
\end{tikzpicture}

&

\begin{tikzpicture}[scale=0.35]
\tikzstyle{black} = [circle, minimum width=2pt, fill, inner sep=0pt]
\foreach \x in {1,...,6}
\node[black] (n\x) at (\x,0) {};
\draw (n1) to [out=90,in=90] (n6);
\draw (n2) to [out=90,in=90] (n3);
\draw (n4) to [out=90,in=90] (n5);
\end{tikzpicture}

&

\begin{tikzpicture}[scale=0.35]
\tikzstyle{black} = [circle, minimum width=2pt, fill, inner sep=0pt]
\foreach \x in {1,...,6}
\node[black] (n\x) at (\x,0) {};
\draw (n1) to [out=90,in=90] (n4);
\draw (n2) to [out=90,in=90] (n3);
\draw (n5) to [out=90,in=90] (n6);
\end{tikzpicture}

&

\begin{tikzpicture}[scale=0.35]
\tikzstyle{black} = [circle, minimum width=2pt, fill, inner sep=0pt]
\foreach \x in {1,...,6}
\node[black] (n\x) at (\x,0) {};
\draw (n1) to [out=90,in=90] (n2);
\draw (n3) to [out=90,in=90] (n6);
\draw (n4) to [out=90,in=90] (n5);
\end{tikzpicture}

&

\begin{tikzpicture}[scale=0.35]
\tikzstyle{black} = [circle, minimum width=2pt, fill, inner sep=0pt]
\foreach \x in {1,...,6}
\node[black] (n\x) at (\x,0) {};
\draw (n1) to [out=90,in=90] (n2);
\draw (n3) to [out=90,in=90] (n4);
\draw (n5) to [out=90,in=90] (n6);
\end{tikzpicture}

\\


\begin{tikzpicture}[scale=0.35]
\tikzstyle{black} = [circle, minimum width=2pt, fill, inner sep=0pt]
\foreach \x in {1,...,6}
\node[black] (n\x) at (\x,0) {};
\draw (n1) to [out=90,in=90] (n4);
\draw (n2) to [out=90,in=90] (n5);
\draw (n3) to [out=90,in=90] (n6);
\end{tikzpicture}

&

\begin{tikzpicture}[scale=0.35]
\tikzstyle{black} = [circle, minimum width=2pt, fill, inner sep=0pt]
\foreach \x in {1,...,6}
\node[black] (n\x) at (\x,0) {};
\draw (n1) to [out=90,in=90] (n3);
\draw (n2) to [out=90,in=90] (n5);
\draw (n4) to [out=90,in=90] (n6);
\end{tikzpicture}

&

\begin{tikzpicture}[scale=0.35]
\tikzstyle{black} = [circle, minimum width=2pt, fill, inner sep=0pt]
\foreach \x in {1,...,6}
\node[black] (n\x) at (\x,0) {};
\draw (n1) to [out=90,in=90] (n3);
\draw (n2) to [out=90,in=90] (n4);
\draw (n5) to [out=90,in=90] (n6);
\end{tikzpicture}

&

\begin{tikzpicture}[scale=0.35]
\tikzstyle{black} = [circle, minimum width=2pt, fill, inner sep=0pt]
\foreach \x in {1,...,6}
\node[black] (n\x) at (\x,0) {};
\draw (n1) to [out=90,in=90] (n2);
\draw (n3) to [out=90,in=90] (n5);
\draw (n4) to [out=90,in=90] (n6);
\end{tikzpicture}

&

\begin{tikzpicture}[scale=0.35]
\tikzstyle{black} = [circle, minimum width=2pt, fill, inner sep=0pt]
\foreach \x in {1,...,6}
\node[black] (n\x) at (\x,0) {};
\draw (n1) to [out=90,in=90] (n5);
\draw (n2) to [out=90,in=90] (n6);
\draw (n3) to [out=90,in=90] (n4);
\end{tikzpicture}

\\


\begin{tikzpicture}[scale=0.35]
\tikzstyle{black} = [circle, minimum width=2pt, fill, inner sep=0pt]
\foreach \x in {1,...,6}
\node[black] (n\x) at (\x,0) {};
\draw (n1) to [out=90,in=90] (n4);
\draw (n2) to [out=90,in=90] (n6);
\draw (n3) to [out=90,in=90] (n5);
\end{tikzpicture}

&

\begin{tikzpicture}[scale=0.35]
\tikzstyle{black} = [circle, minimum width=2pt, fill, inner sep=0pt]
\foreach \x in {1,...,6}
\node[black] (n\x) at (\x,0) {};
\draw (n1) to [out=90,in=90] (n5);
\draw (n2) to [out=90,in=90] (n4);
\draw (n3) to [out=90,in=90] (n6);
\end{tikzpicture}

&
\begin{tikzpicture}[scale=0.35]
\tikzstyle{black} = [circle, minimum width=2pt, fill, inner sep=0pt]
\foreach \x in {1,...,6}
\node[black] (n\x) at (\x,0) {};
\draw (n1) to [out=90,in=90] (n3);
\draw (n2) to [out=90,in=90] (n6);
\draw (n4) to [out=90,in=90] (n5);
\end{tikzpicture}

&

\begin{tikzpicture}[scale=0.35]
\tikzstyle{black} = [circle, minimum width=2pt, fill, inner sep=0pt]
\foreach \x in {1,...,6}
\node[black] (n\x) at (\x,0) {};
\draw (n1) to [out=90,in=90] (n5);
\draw (n2) to [out=90,in=90] (n3);
\draw (n4) to [out=90,in=90] (n6);
\end{tikzpicture}

&

\begin{tikzpicture}[scale=0.35]
\tikzstyle{black} = [circle, minimum width=2pt, fill, inner sep=0pt]
\foreach \x in {1,...,6}
\node[black] (n\x) at (\x,0) {};
\draw (n1) to [out=90,in=90] (n6);
\draw (n2) to [out=90,in=90] (n4);
\draw (n3) to [out=90,in=90] (n5);
\end{tikzpicture}

\end{tabular} 
\end{center}
\caption{The set of all linear chord diagrams with three chords each of length at least one.}
\label{linchordn3}
\end{figure}

Sullivan gave a relationship for the diagonals of the array (\ref{Sultable}) below of linear chord diagrams with $n$ chords where all chords have length at least $k$. See \cite{Sul}. 

\begin{equation}
\begin{array}{c|ccccc}
n \backslash k & 1 & 2 & 3 & 4 & 5 \\\hline
1 & 1 &&&&\\ 
2 & 3 & 1 &&&\\
3 & 15 & 5 & 1 &&\\
4 & 105 & 36 & 10 & 1\\
5 & 945 & 329 & 99 & 20 & 1\\
\vdots &&&\cdots&&\\
\end{array}
\label{Sultable}
\end{equation}
We are interested in further refining this table by considering the number of chords of length $k$ for each linear chord diagram with $n$ chords.

In this paper, we will focus only on the case where $k=1$, i.e. those linear chord diagrams with $n$ chords whose shortest chord is of length at least $1$, classified by the number of chords of length $1$.  In other words, we will look at the set of all linear chord diagrams with $n$ chords (since all diagrams have a shortest chord of length at least $1$) by number of chords of length $1$. In this context, a short chord will be understood to be a chord of length one. Table (\ref{shortchordtriangle}) gives\ignore{the numbers $$L_{n,s}:=\left|\{\pi\in \mathcal{L}_{n}^{1}\  \vert \ sc(\pi)=s\}\right|,$$ i.e.,} the total number $L_{n,s}$ of linear chord diagrams with $n$ chords exactly $s$ of which are short (have length one):
\begin{equation}
\begin{array}{c|ccccccccc}
n \backslash s & 0 & 1 & 2 & 3 & 4 & 5 & 6 & 7 & 8\\\hline
0 & 1 &&&&&&&&\\
1 & 0 & 1 &&&&&&&\\ 
2 & 1 & 1 & 1 &&&&&&\\
3 & 5 & 6 & 3 & 1 &&&&&\\
4 & 36 & 41 & 21 & 6 & 1&&&&\\
5 & 329 & 365 & 185 & 55 & 10 & 1 &&&\\
6 & 3655 & 3984 & 2010 & 610 & 120 & 15 & 1&& \\
7 & 47844 & 51499 & 25914 & 7980 & 1645 & 231 & 21 & 1&\\
8 & 72135 & 769159 & 386407 & 120274 & 25585 & 3850 & 406 & 28 & 1\\
\vdots &&&\cdots&&&&&&\\
\end{array}
\label{shortchordtriangle}
\end{equation}

More generally, if we let $\mathcal{L}_{n}^k$ denote the set of all linear chord diagrams having $n$ chords, each of which has length at least $k$, then for $\pi\in  \mathcal{L}_{n}^k$, we refer to a chord of length $k$ in $\pi$ as a {\em short chord}. For $\pi\in \mathcal{L}_{n}^{k}$, define $sc(\pi)$ to be the number of short chords, i.e. the number of chords of length $k$, in $\pi$.  We want to consider the distribution of the number of short chords on $\mathcal{L}_{n}^{k}$, that is,
\begin{displaymath} L_{n}^{k}(q):=\sum_{\pi \in \mathcal{L}_{n}^{k}}{q^{sc(\pi)}}.\end{displaymath} See Figure \ref{Sullivantriangle}. Note that the numbers $L_{n,s}$ in the table (\ref{shortchordtriangle}) above are the coefficients of the polynomials $L_n^1(q)$ in the first column of the table in Figure \ref{Sullivantriangle}. The numbers $L_{n,s}$ are recorded as \href{https://oeis.org/A079267}{A079267} in \cite{OEIS}.

\begin{figure}
\centering
$\begin{array}{ >{\arraybackslash$} p{0.1in} <{$} | >{\arraybackslash$} P{1.5in} <{$}  >{\arraybackslash$} P{1.5in} <{$}  >{\arraybackslash$} p{1.5in} <{$} }
n & L_{n}^{1}(q) & {L}_{n}^{2}(q) & {L}_{n}^{3}(q) \\\hline
1 & q &  & \\ 
2 & 1+q+q^2 & q^2 &  \\
3 & 5+6q+3q^2+q^3 & 1+2q+2q^2 & q^3\\
4 & 36+41q+21q^2+6q^3+q^4 & 10 +14q+9q^2+2q^3+q^4 & 1+3q+4q^2+2q^3\\
\vdots &\vdots &&\\
\end{array}$
\caption{The distribution of the number of minimal length chords on $\mathcal{L}_{n}^{k}$ for $k=1,2,3,\dots$.}
\label{Sullivantriangle}
\end{figure}

It is known that the Catalan numbers $C_n=\frac{1}{n+1}\binom{2n}{n}$ count the number of noncrossing matchings on $[2n]$. The Narayana numbers, $N(n,k)= \frac{1}{k}\binom{n-1}{k-1}\binom{n}{k-1}$, count the number of noncrossing matchings on $[2n]$ with $k$ arcs of the form $(i,i+1)$ (i.e., $k$ chords of length $1$).  The Narayana numbers are given by the following table:
\begin{displaymath}
\begin{array}{c|ccccc}
n \backslash k & 1 & 2 & 3 & 4 & 5 \\\hline
1 & 1 &&&&\\ 
2 & 1 & 1 &&&\\
3 & 1 & 3 & 1 &&\\
4 & 1 & 6 & 6 & 1\\
5 & 1 & 10 & 20 & 10 & 1\\
\vdots &&&\cdots&&\\
\end{array}
\end{displaymath}
We may consider table (\ref{shortchordtriangle}) as a generalization of the Narayana numbers on the set of all matchings on $[2n]$. 

 An {\em LR pair} in a linear chord diagram is a pair of consecutive integers $(i, i+1)$ where $i$ is a startpoint of a chord and $i+1$ is an endpoint of a (possibly different) chord. The Narayana numbers also count the number of nonnesting matchings on $[2n]$ with $k$ LR pairs, which can be seen as follows.  We first make the well known observation that the Narayana numbers count the number of Dyck paths of length $2n$ with $k$ peaks, as described by \cite{Sta}. 
 Consider a Dyck path as a list of $n$ up steps $U$ and $n$ down steps $D$, such that at any step the number of preceding down steps never exceeds the number of preceding up steps. Label the $U$ steps of the list with the numbers $1$ through $n$, from left to right, and label the $D$ steps of the list with the numbers $1$ through $n$, from left to right. Draw a linear chord diagram with $n$ chords by connecting the $U$ step labeled $i$ with the $D$ step labeled $i$, for all $i=1,\dots,n$. The result is a unique nonnesting linear chord diagram with $n$ chords and $k$ LR-pairs.  

We obtain a different generalization of the Narayana numbers by considering LR pairs in all matchings on $[2n]$ (equivalently, linear chord diagrams with $n$ chords). Let $T_{n,k}$ denote the number of all linear chord diagrams with $n$ chords having exactly $k$ LR pairs. The following table gives the numbers $T_{n,k}$:
\begin{equation}
\begin{array}{c|cccccccc}
n \backslash k  & 1 & 2 & 3 & 4 & 5 & 6 & 7 & \\\hline
1 & 1 &&&&&&&\\ 
2 & 2 & 1 &&&&&&\\
3 & 6 & 8 & 1 &&&&&\\
4 & 24 & 58 & 22 & 1 &&&&\\
5 & 120 & 444 & 328 & 52 & 1 &&&\\
6 & 720 & 3708 & 4400 & 1452 & 114 & 1 && \\
7 & 5040 & 33984 & 58140 & 32120 & 5610 & 240 & 1 &\\
\vdots &&&\cdots&&&&&\\
\end{array}
\label{LRtriangle}
\end{equation}

In Section 2 of this paper, we prove combinatorially that the rows of table (\ref{shortchordtriangle}) are unimodal and conjecture that they are also log-concave.  We also give the exponential generating functions for the columns of this triangle. It is known that the expected number of short chords among the elements of $\mathcal{L}_n^1$ is one. This fact was first shown by \cite{Kre} and more recently extended by \cite{You}. However, in Section 2, we provide a combinatorial proof of this fact by providing a bijection between the total number of linear chord diagrams with $n$ chords and the total number of chords of length one among all linear chord diagrams with $n$ chords. In Section 3, we explore the triangle given by table (\ref{LRtriangle}). We connect table (\ref{LRtriangle}) to the second order Eulerian triangle and prove that the rows of table (\ref{LRtriangle}) are log-concave (and thus also unimodal).

\section{A Generalized Narayana Triangle}
We now explore several interesting properties of the generalized Narayana triangle given in table (\ref{shortchordtriangle}), i.e, the number of linear chord diagrams with $n$ chords, exactly $s$ of which have length $1$.  It is easy to see that the entries on the main diagonal of this triangle will always have the value 1.  Furthermore, the entries along the first sub-diagonal correspond to the triangular numbers $\binom{n}{2}$.

Now recall from Section \ref{background} that $L_{n,s}=\left|\{\pi\in \mathcal{L}_{n}^{1}\  \vert \ sc(\pi)=s\}\right|$ is the number of linear chord diagrams having $n$ chords exactly $s$ of which are chords of length one. We introduce the following shorthand, \begin{displaymath} {L}_n(q):=L_n^1(q)=\sum_{s=0}^n{L_{n,s} \ q^s   } , \end{displaymath}
and note that the coefficients of ${L}_n(q)$ are given by the rows of table (\ref{shortchordtriangle}).

\begin{theorem}
${L}_n(q)$ is unimodal.
\end{theorem}

\begin{proof}
Let $\mathcal{L}_{n,s}:=\{\pi\in \mathcal{L}_n^1 \vert\ sc(\pi)=s\}$ denote the set of linear chord diagrams having $n$ chords and exactly $s$ short chords (i.e., chords of length one), so that $\left|\mathcal{L}_{n,s}\right| =L_{n,s}.$ We can show that ${L}_n(q)$ is unimodal by establishing injective maps $\phi_1: \mathcal{L}_{n,0}\to \mathcal{L}_{n,1}$ and $\phi_j:\mathcal{L}_{n,j}\to \mathcal{L}_{n,j-1}$ for $j=2,\dots, n$.

Define $\phi_j: \mathcal{L}_{n,j}\to \mathcal{L}_{n,j-1}$ for $j\geq 2$ as follows.  Let $\pi \in \mathcal{L}_{n,j}$.  Then $\pi$ has $j$ short chords, where $j\geq 2$.  Take the rightmost short chord, say $\{i,i+1\}$, and unwrap this chord by sending the right endpoint to the beginning of the diagram, i.e., replacing it with the long chord $\{1,i\}$.  The result is a linear chord diagram with $j-1$ short chords where all short chords are covered.  $\phi_j$ is injective.  Define $\phi_1: \mathcal{L}_{n,0}\to \mathcal{L}_{n,1}$ as follows.  Let $\pi \in \mathcal{L}_{n,0}$.  Take the first (long) chord $\{1,i\}$ and turn it into the short chord $\{i,i+1\}$; the result is a chord diagram with exactly one short chord.  This map is injective but not surjective because it will not produce the chord diagrams in $\mathcal{L}_{n,1}$ that begin with a short chord.
\end{proof}

\begin{conjecture}
The coefficients of $L_{n}(q)$ form a log-concave sequence. 
\end{conjecture}

\begin{theorem}
The exponential generating function for the number $L_{n,s}$ of linear chord diagrams with $n$ chords, exactly $s$ of which are chords of length one, is \begin{displaymath}\frac{e^{(-1+\sqrt{1-2t})}}{\sqrt{1-2t}}\cdot \frac{(1-\sqrt{1-2t})^s}{s!}.\end{displaymath}
\end{theorem}

\begin{proof}
First we note that, for $s\geq 0$, the numbers $L_{n,s}$ satisfy the recurrence
\begin{equation}
L_{n,s}=L_{n-1,s-1}+(2n-2-s)L_{n-1,s}+(s+1)L_{n-1,s+1}
\label{Lrecurrence}
\end{equation}
 with $L_{0,0}=1$ and $L_{n,s}=0$ for $s>n$.  A combinatorial proof of this recurrence is given in \cite{KO}.
We want to find the exponential generating function \begin{displaymath}L_s(t):=\sum_{n=0}^\infty{L_{n,s}\frac{t^n}{n!}}\end{displaymath} for the $s$th column of the triangle in (\ref{shortchordtriangle}). To do this, we extend the approach in \cite{KO} that was used to find the exponential generating function $L_0(t).$ Recalling this approach, we define \begin{displaymath}L_n(q):=L_{n,0}+L_{n,1}q+L_{n,2}q^2+\cdots+L_{n,n}q^n\end{displaymath} and note that \begin{displaymath}L_{n,s}=\frac{L_n^{(s)}(0)}{s!}.\end{displaymath} In \cite{KO}, the authors define the bivariate generating function \begin{displaymath}\omega(q,t)=\sum_{n=0}^\infty{L_n(q) \frac{t^n}{n!}}\end{displaymath} and use recurrence (\ref{Lrecurrence}) to show that \begin{displaymath}\omega(q,t)=\frac{e^{(1-q)(-1+\sqrt{1-2t})}}{\sqrt{1-2t}},\end{displaymath} thereby setting $q=0$ to obtain \begin{displaymath}L_0(t)=\frac{e^{(-1+\sqrt{1-2t})}}{\sqrt{1-2t}}\end{displaymath} as the exponential generating function for the first column of (\ref{shortchordtriangle}). 

We can extend this observation by noting that \begin{displaymath}L_s(t)=\sum_{n=0}^\infty{L_{n,s}\frac{t^n}{n!}}=  \sum_{n=0}^\infty{ \frac{L_n^{(s)}(0)}{s!} \frac{t^n}{n!} } = \left(\frac{1}{s!}  \left. \frac{\partial^s\omega}{\partial q^s}  \right\vert_{q=0}\right)\end{displaymath} and \begin{displaymath} \frac{1}{s!} \frac{\partial^s\omega}{\partial q^s}= \frac{e^{(1-q)(-1+\sqrt{1-2t})}}{\sqrt{1-2t}}\cdot \frac{(1-\sqrt{1-2t})^s}{s!}.\end{displaymath} Letting $q=0$ in the last equation we obtain \begin{displaymath}L_s(t)=\frac{e^{(-1+\sqrt{1-2t})}}{\sqrt{1-2t}}\cdot \frac{(1-\sqrt{1-2t})^s}{s!}.\end{displaymath}
\end{proof}

The last line of the preceding proof implies that the triangle in (\ref{shortchordtriangle}) is an {\em exponential Riordan array} with {\em initial function} $g=\frac{e^{(-1+\sqrt{1-2t})}}{\sqrt{1-2t}}$ and {\em multiplier function} $f=1-\sqrt{1-2t}.$ For reference, see \cite{Shap} or \cite{Bar}. Furthermore, we may use Riordan group algebra to count the total number of short chords among all linear chord diagrams with $n$ chords. To do this, we multiply triangle (\ref{shortchordtriangle}) by the infinite column vector $(0,1,2,3,\dots)^T$, using Riordan group multiplication. Since $te^t$ is the exponential generation function for the sequence $0, 1,2, 3, 4,\dots$, the Riordan multiplication proceeds as follows:
\begin{eqnarray*}
\left(\frac{e^{(-1+\sqrt{1-2t})}}{\sqrt{1-2t}}, 1-\sqrt{1-2t} \right) \ast t e^t &=& \frac{e^{(-1+\sqrt{1-2t})}}{\sqrt{1-2t}} \left(1-\sqrt{1-2t}\right) e^{1-\sqrt{1-2t}}\\
&=& \frac{1}{\sqrt{1-2t}}-1\\
&=& \sum_{n=0}^\infty{ (2n-1)!! \frac{t^n}{n!} },
\end{eqnarray*}
where $(2n-1)!!=1\cdot 3\cdot 5\cdot (2n-1)$ for $n\geq 1$ and $0$ otherwise. The result is that the total number of short chords among all linear chord diagrams with $n$ chords is the same as the total number of linear chord diagrams with $n$ chords. 

We now provide a bijective argument for the fact that the expected number of short chords among all linear chord diagrams with $n$ chords is one.
\begin{theorem}[\cite{Kre}]\label{shortbijection}
The total number of short chords among all linear chord diagrams with $n$ chords is equal to the number of linear chord diagrams with $n$ chords.
\end{theorem}

\begin{proof}
We will construct a bijection which maps each short chord $s$ in a linear chord diagram $D$ with $n$ chords to a unique linear chord diagram $D_s$ with $n$ chords.

To begin, identify a short chord $s$ in any linear chord diagram $D$ with $n$ chords.  If $s=\{1,2\}$, then $s$ will be mapped to the diagram $D$ that $s$ is a part of, that is, $D_s=D$.  If $s=\{i, i+1\}$ in $D$, where $i>1$, then $s$ will be mapped to the diagram $D_s$ that has the same connectivity as $D$, except that the chord $\{i, i+1\}$ in $D$ has been replaced with the chord $\{1,i+1\}$ and all start and endpoints between $1$ and $i-1$ have moved one position to the right. It should be clear that if $s_1$ and $s_2$ are two different short chords, either from the same or different diagrams, then $D_{s_1}\neq D_{s_2}$, and hence this map is injective.

To invert this map, we do the following. Take a linear chord diagram $D$ and consider its first chord $\{1,i+1\}$. If $i=1$, associate $D$ with its first chord $s=\{1,2\}$. If $i>1$, create a new diagram by removing $\{1,i+1\}$ from $D$, shifting all the start and endpoints $2,\dots,i$ of $D$ one position to the left and inserting the chord $s=\{i,i+1\}$. Associate $D$ with the short chord $s=\{i,i+1\}$ from this new diagram. It should be clear from the description of this inverse map that if $D_1$ and $D_2$ are two different linear chord diagrams, then the short chord $s_1$ associated with $D_1$ will be different from the short chord $s_2$ associated with $D_2$. 

See Figure \ref{shortchordbijectionn=3} for a depiction of this bijection when $n=3$.  

Hence, we have a bijection between the set of short chords among all linear chord diagrams with $n$ chords and the set of linear chord diagrams with $n$ chords.
\end{proof}


\begin{figure}
\begin{center}
\begin{tabular}{|p{0.8in}|p{0.8in}|p{0.8in}|p{0.8in}|p{0.8in}|}\hline
\begin{tikzpicture}[scale=0.35]
\tikzstyle{black} = [circle, minimum width=2pt, fill, inner sep=0pt]
\foreach \x in {1,...,6}
\node[black] (n\x) at (\x,0) {};
\draw (n1) to [out=90,in=90] (n6);
\draw (n2) to [out=90,in=90] (n5);
\draw[ultra thick] (n3) to [out=90,in=90] (n4);
\end{tikzpicture}

&

\begin{tikzpicture}[scale=0.35]
\tikzstyle{black} = [circle, minimum width=2pt, fill, inner sep=0pt]
\foreach \x in {1,...,6}
\node[black] (n\x) at (\x,0) {};
\draw (n1) to [out=90,in=90] (n6);
\draw[ultra thick] (n2) to [out=90,in=90] (n3);
\draw (n4) to [out=90,in=90] (n5);
\end{tikzpicture}

&

\begin{tikzpicture}[scale=0.35]
\tikzstyle{black} = [circle, minimum width=2pt, fill, inner sep=0pt]
\foreach \x in {1,...,6}
\node[black] (n\x) at (\x,0) {};
\draw (n1) to [out=90,in=90] (n6);
\draw (n2) to [out=90,in=90] (n3);
\draw[ultra thick] (n4) to [out=90,in=90] (n5);
\end{tikzpicture}
&

\begin{tikzpicture}[scale=0.35]
\tikzstyle{black} = [circle, minimum width=2pt, fill, inner sep=0pt]
\foreach \x in {1,...,6}
\node[black] (n\x) at (\x,0) {};
\draw (n1) to [out=90,in=90] (n4);
\draw[ultra thick] (n2) to [out=90,in=90] (n3);
\draw (n5) to [out=90,in=90] (n6);
\end{tikzpicture}

&

\begin{tikzpicture}[scale=0.35]
\tikzstyle{black} = [circle, minimum width=2pt, fill, inner sep=0pt]
\foreach \x in {1,...,6}
\node[black] (n\x) at (\x,0) {};
\draw (n1) to [out=90,in=90] (n4);
\draw (n2) to [out=90,in=90] (n3);
\draw[ultra thick] (n5) to [out=90,in=90] (n6);
\end{tikzpicture}
\\


\begin{tikzpicture}[scale=0.35]
\tikzstyle{black} = [circle, minimum width=2pt, fill, inner sep=0pt]
\foreach \x in {1,...,6}
\node[black] (n\x) at (\x,0) {};
\draw (n1) to [out=90,in=90] (n4);
\draw (n2) to [out=90,in=90] (n6);
\draw (n3) to [out=90,in=90] (n5);
\end{tikzpicture}
&

\begin{tikzpicture}[scale=0.35]
\tikzstyle{black} = [circle, minimum width=2pt, fill, inner sep=0pt]
\foreach \x in {1,...,6}
\node[black] (n\x) at (\x,0) {};
\draw (n1) to [out=90,in=90] (n3);
\draw (n2) to [out=90,in=90] (n6);
\draw (n4) to [out=90,in=90] (n5);
\end{tikzpicture}
&

\begin{tikzpicture}[scale=0.35]
\tikzstyle{black} = [circle, minimum width=2pt, fill, inner sep=0pt]
\foreach \x in {1,...,6}
\node[black] (n\x) at (\x,0) {};
\draw (n1) to [out=90,in=90] (n5);
\draw (n2) to [out=90,in=90] (n6);
\draw (n3) to [out=90,in=90] (n4);
\end{tikzpicture}
&

\begin{tikzpicture}[scale=0.35]
\tikzstyle{black} = [circle, minimum width=2pt, fill, inner sep=0pt]
\foreach \x in {1,...,6}
\node[black] (n\x) at (\x,0) {};
\draw (n1) to [out=90,in=90] (n3);
\draw (n2) to [out=90,in=90] (n4);
\draw (n5) to [out=90,in=90] (n6);
\end{tikzpicture}
&

\begin{tikzpicture}[scale=0.35]
\tikzstyle{black} = [circle, minimum width=2pt, fill, inner sep=0pt]
\foreach \x in {1,...,6}
\node[black] (n\x) at (\x,0) {};
\draw (n1) to [out=90,in=90] (n6);
\draw (n2) to [out=90,in=90] (n5);
\draw (n3) to [out=90,in=90] (n4);
\end{tikzpicture}

\\\hline


\begin{tikzpicture}[scale=0.35]
\tikzstyle{black} = [circle, minimum width=2pt, fill, inner sep=0pt]
\foreach \x in {1,...,6}
\node[black] (n\x) at (\x,0) {};
\draw[ultra thick] (n1) to [out=90,in=90] (n2);
\draw (n3) to [out=90,in=90] (n6);
\draw (n4) to [out=90,in=90] (n5);
\end{tikzpicture}
&

\begin{tikzpicture}[scale=0.35]
\tikzstyle{black} = [circle, minimum width=2pt, fill, inner sep=0pt]
\foreach \x in {1,...,6}
\node[black] (n\x) at (\x,0) {};
\draw (n1) to [out=90,in=90] (n2);
\draw (n3) to [out=90,in=90] (n6);
\draw[ultra thick] (n4) to [out=90,in=90] (n5);
\end{tikzpicture}
&

\begin{tikzpicture}[scale=0.35]
\tikzstyle{black} = [circle, minimum width=2pt, fill, inner sep=0pt]
\foreach \x in {1,...,6}
\node[black] (n\x) at (\x,0) {};
\draw[ultra thick] (n1) to [out=90,in=90] (n2);
\draw (n3) to [out=90,in=90] (n4);
\draw (n5) to [out=90,in=90] (n6);
\end{tikzpicture}
&

\begin{tikzpicture}[scale=0.35]
\tikzstyle{black} = [circle, minimum width=2pt, fill, inner sep=0pt]
\foreach \x in {1,...,6}
\node[black] (n\x) at (\x,0) {};
\draw (n1) to [out=90,in=90] (n2);
\draw[ultra thick] (n3) to [out=90,in=90] (n4);
\draw (n5) to [out=90,in=90] (n6);
\end{tikzpicture}
&

\begin{tikzpicture}[scale=0.35]
\tikzstyle{black} = [circle, minimum width=2pt, fill, inner sep=0pt]
\foreach \x in {1,...,6}
\node[black] (n\x) at (\x,0) {};
\draw (n1) to [out=90,in=90] (n2);
\draw (n3) to [out=90,in=90] (n4);
\draw[ultra thick] (n5) to [out=90,in=90] (n6);
\end{tikzpicture}
\\


\begin{tikzpicture}[scale=0.35]
\tikzstyle{black} = [circle, minimum width=2pt, fill, inner sep=0pt]
\foreach \x in {1,...,6}
\node[black] (n\x) at (\x,0) {};
\draw (n1) to [out=90,in=90] (n2);
\draw (n3) to [out=90,in=90] (n6);
\draw (n4) to [out=90,in=90] (n5);
\end{tikzpicture}
&

\begin{tikzpicture}[scale=0.35]
\tikzstyle{black} = [circle, minimum width=2pt, fill, inner sep=0pt]
\foreach \x in {1,...,6}
\node[black] (n\x) at (\x,0) {};
\draw (n1) to [out=90,in=90] (n5);
\draw (n2) to [out=90,in=90] (n3);
\draw (n4) to [out=90,in=90] (n6);
\end{tikzpicture}
&

\begin{tikzpicture}[scale=0.35]
\tikzstyle{black} = [circle, minimum width=2pt, fill, inner sep=0pt]
\foreach \x in {1,...,6}
\node[black] (n\x) at (\x,0) {};
\draw (n1) to [out=90,in=90] (n2);
\draw (n3) to [out=90,in=90] (n4);
\draw (n5) to [out=90,in=90] (n6);
\end{tikzpicture}
&

\begin{tikzpicture}[scale=0.35]
\tikzstyle{black} = [circle, minimum width=2pt, fill, inner sep=0pt]
\foreach \x in {1,...,6}
\node[black] (n\x) at (\x,0) {};
\draw (n1) to [out=90,in=90] (n4);
\draw (n2) to [out=90,in=90] (n3);
\draw (n5) to [out=90,in=90] (n6);
\end{tikzpicture}
&

\begin{tikzpicture}[scale=0.35]
\tikzstyle{black} = [circle, minimum width=2pt, fill, inner sep=0pt]
\foreach \x in {1,...,6}
\node[black] (n\x) at (\x,0) {};
\draw (n1) to [out=90,in=90] (n6);
\draw (n2) to [out=90,in=90] (n3);
\draw (n4) to [out=90,in=90] (n5);
\end{tikzpicture}

\\\hline


\begin{tikzpicture}[scale=0.35]
\tikzstyle{black} = [circle, minimum width=2pt, fill, inner sep=0pt]
\foreach \x in {1,...,6}
\node[black] (n\x) at (\x,0) {};
\draw (n1) to [out=90,in=90] (n3);
\draw (n2) to [out=90,in=90] (n4);
\draw[ultra thick] (n5) to [out=90,in=90] (n6);
\end{tikzpicture}
&

\begin{tikzpicture}[scale=0.35]
\tikzstyle{black} = [circle, minimum width=2pt, fill, inner sep=0pt]
\foreach \x in {1,...,6}
\node[black] (n\x) at (\x,0) {};
\draw[ultra thick] (n1) to [out=90,in=90] (n2);
\draw (n3) to [out=90,in=90] (n5);
\draw (n4) to [out=90,in=90] (n6);
\end{tikzpicture}
&

\begin{tikzpicture}[scale=0.35]
\tikzstyle{black} = [circle, minimum width=2pt, fill, inner sep=0pt]
\foreach \x in {1,...,6}
\node[black] (n\x) at (\x,0) {};
\draw (n1) to [out=90,in=90] (n5);
\draw (n2) to [out=90,in=90] (n6);
\draw[ultra thick] (n3) to [out=90,in=90] (n4);
\end{tikzpicture}
&

\begin{tikzpicture}[scale=0.35]
\tikzstyle{black} = [circle, minimum width=2pt, fill, inner sep=0pt]
\foreach \x in {1,...,6}
\node[black] (n\x) at (\x,0) {};
\draw (n1) to [out=90,in=90] (n3);
\draw (n2) to [out=90,in=90] (n6);
\draw[ultra thick] (n4) to [out=90,in=90] (n5);
\end{tikzpicture}

&

\begin{tikzpicture}[scale=0.35]
\tikzstyle{black} = [circle, minimum width=2pt, fill, inner sep=0pt]
\foreach \x in {1,...,6}
\node[black] (n\x) at (\x,0) {};
\draw (n1) to [out=90,in=90] (n5);
\draw[ultra thick] (n2) to [out=90,in=90] (n3);
\draw (n4) to [out=90,in=90] (n6);
\end{tikzpicture}
\\


\begin{tikzpicture}[scale=0.35]
\tikzstyle{black} = [circle, minimum width=2pt, fill, inner sep=0pt]
\foreach \x in {1,...,6}
\node[black] (n\x) at (\x,0) {};
\draw (n1) to [out=90,in=90] (n6);
\draw (n2) to [out=90,in=90] (n4);
\draw (n3) to [out=90,in=90] (n5);
\end{tikzpicture}

&

\begin{tikzpicture}[scale=0.35]
\tikzstyle{black} = [circle, minimum width=2pt, fill, inner sep=0pt]
\foreach \x in {1,...,6}
\node[black] (n\x) at (\x,0) {};
\draw (n1) to [out=90,in=90] (n2);
\draw (n3) to [out=90,in=90] (n5);
\draw (n4) to [out=90,in=90] (n6);
\end{tikzpicture}

&

\begin{tikzpicture}[scale=0.35]
\tikzstyle{black} = [circle, minimum width=2pt, fill, inner sep=0pt]
\foreach \x in {1,...,6}
\node[black] (n\x) at (\x,0) {};
\draw (n1) to [out=90,in=90] (n4);
\draw (n2) to [out=90,in=90] (n5);
\draw (n3) to [out=90,in=90] (n6);
\end{tikzpicture}

&

\begin{tikzpicture}[scale=0.35]
\tikzstyle{black} = [circle, minimum width=2pt, fill, inner sep=0pt]
\foreach \x in {1,...,6}
\node[black] (n\x) at (\x,0) {};
\draw (n1) to [out=90,in=90] (n5);
\draw (n2) to [out=90,in=90] (n4);
\draw (n3) to [out=90,in=90] (n6);
\end{tikzpicture}

&
\begin{tikzpicture}[scale=0.35]
\tikzstyle{black} = [circle, minimum width=2pt, fill, inner sep=0pt]
\foreach \x in {1,...,6}
\node[black] (n\x) at (\x,0) {};
\draw (n1) to [out=90,in=90] (n3);
\draw (n2) to [out=90,in=90] (n5);
\draw (n4) to [out=90,in=90] (n6);
\end{tikzpicture}

\\\hline

\end{tabular} 
\end{center}
\caption{An illustration of the bijective proof of Theorem \ref{shortbijection} in the case where $n=3$. For each short chord $s$ (highlighted in bold) above, see its corresponding linear chord diagram $D_s$ pictured below.}
\label{shortchordbijectionn=3}
\end{figure}
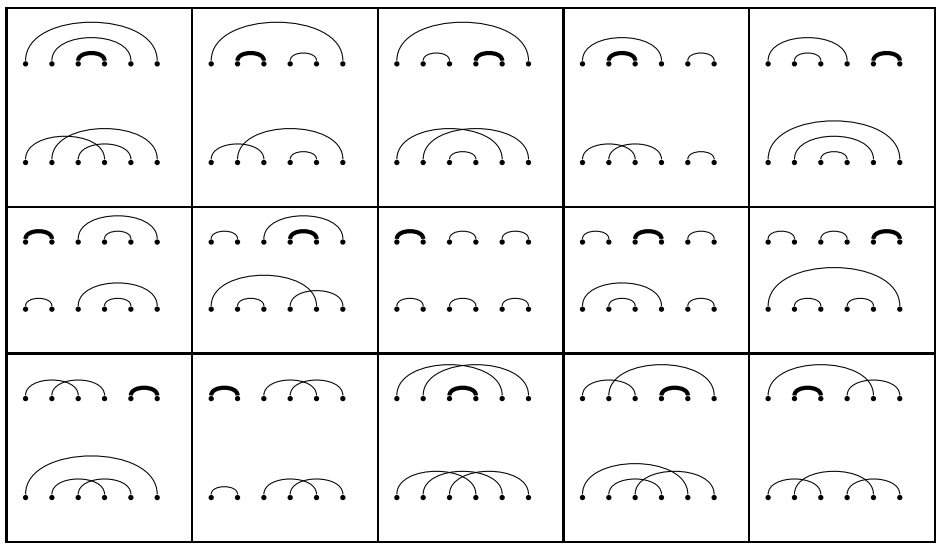

\section{The Second-Order Eulerian Triangle}

In this section, we explore the second-order Eulerian triangle, which can be thought of as another generalization of the Narayana triangle.  The second-order Eulerian triangle is given below, where entry $E(n,k)$ is known to count the number of permutations of the multiset $\{1,1,2,2,\dots,n,n\}$ with $k$ ascents such that between any two copies of $m$ there are only numbers less than $m$.  
\begin{displaymath}
\begin{array}{c|cccccccc}
n \backslash k & 0 & 1 & 2 & 3 & 4 & 5 & 6 &  \\\hline
1 & 1 &&&&&&&\\ 
2 & 1 & 2 &&&&&&\\
3 & 1 & 8 & 6 &&&&&\\
4 & 1 & 22 & 58 & 24 &&&&\\
5 & 1 & 52 & 328 & 444 & 120 &&&\\
6 & 1 & 114 & 1452 & 4400 & 3708 & 720 && \\
7 & 1 & 240 & 5610 & 32120 & 58140 & 33984 & 5040 &\\
\vdots &&&\cdots&&&&&\\
\end{array}
\end{displaymath}
This triangle is recorded as \href{https://oeis.org/A008517}{A008517} \cite{OEIS}. The second-order Eulerian numbers are known to satisfy the recurrence:
\begin{equation}
E(n,k) = (k+1)E(n-1,k) + (2n-k-1)E(n-1,k-1)
\label{2ndordereulerian}
\end{equation}
where $E(n,0)=1$. This can be seen by the following argument.

One can form all permutations of the multiset $\{1,1,2,2,\dots,n,n\}$ with $k$ ascents in which between any two copies of $m$ there are only numbers less than $m$ by taking all such permutations of the multiset $\{1,1,2,2,\dots,n-1,n-1\}$ with $k$ ascents and first replacing all numbers $m=1,2,\dots, n-1$ with $m+1$.  Then insert the pair $1 \ 1$ at the end of the permutation or between any ascent, which can be done in $k+1$ ways.  One can also form such a permutation by taking any such permutation of the multiset $\{1,1,2,2,\dots,n-1,n-1\}$ with $k-1$ ascents and inserting the pair $1 \ 1$ at any place which is not an ascent, which can be done in $2n-k-1$ ways.  

The row reversal of the second-order Eulerian triangle gives the triangle

\begin{equation}
\begin{array}{c|cccccccc}
n \backslash k  & 1 & 2 & 3 & 4 & 5 & 6 & 7 & \\\hline
1 & 1 &&&&&&&\\ 
2 & 2 & 1 &&&&&&\\
3 & 6 & 8 & 1 &&&&&\\
4 & 24 & 58 & 22 & 1 &&&&\\
5 & 120 & 444 & 328 & 52 & 1 &&&\\
6 & 720 & 3708 & 4400 & 1452 & 114 & 1 && \\
7 & 5040 & 33984 & 58140 & 32120 & 5610 & 240 & 1 &\\
\vdots &&&\cdots&&&&&\\
\end{array}
\label{eulerianLR}
\end{equation}
whose entries we will denote by $T(n,k)$. Note we have re-indexed the columns to initialize with $k=1.$

\begin{proposition}
The entries $T(n,k)$ of the row reversed second-order Eulerian triangle (\ref{eulerianLR}), which satisfy \begin{equation}T(n,k) = (n-k+1) T(n-1,k-1) + (n-1+k) T(n-1,k)\label{revEulerianrecurrence} \end{equation} with $n,k \geq 1$ and $T(n,1)=n!$, count the number of linear chord diagrams with $n$ chords and $k$ LR pairs.  
\end{proposition}

\begin{proof}
The fact that $T(n,k)$ satisfies recurrence (\ref{revEulerianrecurrence}) follows directly from recurrence (\ref{2ndordereulerian}) for the second-order Eulerian triangle. We will show that the number of linear chord diagrams with $n$ chords and $k$ LR pairs also satisfies recurrence (\ref{revEulerianrecurrence}).

To form a linear chord diagram with $n$ chords and $k$ LR pairs, start with a linear chord diagram with $n-1$ chords and $k-1$ LR pairs.  Place a new start point for a chord at the beginning of the diagram and place the end point for the chord after any start point that isn't in an LR pair.  This can be done in $n-(k-1) = n-k+1$ ways.  One can also form a linear chord diagram with $n$ chords and $k$ LR pairs by starting with a linear chord diagram with $n-1$ chords and $k$ LR pairs.  Place a new start point for a chord at the beginning of the diagram and place the end point for the chord after any end point or after any start point that is in an LR pair.  This can be done in $(n-1)+k$ ways.
\end{proof}

Thus we can consider the row-reversed second-order Eulerian triangle to be a generalization of the Narayana triangle for all linear chords diagrams (not just those that are non-nesting, which are those counted by the Narayana triangle).  Since the generalized Narayana triangle studied in Section 2 has entries that we have proven to be unimodal and conjecture to be log-concave, one might wonder if those properties hold for this second generalization of the Narayana triangle as well.  

Using the following lemma due to \cite{Ku}, we may conclude that the coefficients in each row of the (row-reversed) second-order Eulerian triangle form a log-concave, and therefore unimodal, sequence.

\begin{lemma}[\cite{Ku}]\label{kurtz}
Suppose $\displaystyle\sum_{k=0}^n{R(n,k)q^k}$ is a polynomial for which $R(n,k)$ satisfies the recurrence relation 

\begin{eqnarray*}
 R(n,k) & = & (a_1n+a_2k+a_3)\ R(n-1,k)\\
 && +\ (b_1n+b_2k+b_3)\ R(n-1, k-1),
\end{eqnarray*}

for $n\geq k\geq 1$, with boundary conditions $R(0,0)>0$,
\begin{displaymath}R(n,-1)=R(n,n+1)=0, \ \text{for } n\geq 1,\end{displaymath}
\begin{displaymath}a_1\geq 0,\ a_1+a_2\geq 0,\ a_1+a_2+a_3>0,\ \text{and}\end{displaymath}
\begin{displaymath}b_1\geq 0,\ b_1+b_2\geq 0,\ b_1+b_2+b_3>0.\end{displaymath}

Then, for given $n$, the sequence $\{R(n,k)\}_{0\leq k\leq n}$ is log-concave.
\end{lemma}

\begin{theorem}
 $T(n,k)$ forms a log-concave sequence.
\end{theorem}
\begin{proof}
$T(n,k)$ satisfies Lemma \ref{kurtz} with $a_1=1,a_2=1,a_3=-1,b_1=1,b_2=-1,b_3=1.$
\end{proof}

\section{Summary and Future Work}

We have given two generalizations of the Narayana numbers, the number of linear chord diagrams with $n$ chords and $k$ short chords and the number of linear chord diagrams with $n$ chords and $k$ LR pairs which give the Narayana numbers when restricted to non-crossing linear chord diagrams (in the first case) and nonnesting linear chord diagrams (in the second case), but which are not themselves equidistributed.  In addition, the first generalization gives a study of the coefficients given by the polynomials in the first column of table in Figure \ref{shortchordtriangle}.  

Since the results of this paper focus on generalizing the Narayana numbers and the first column of the table in Figure \ref{shortchordtriangle}, one might wonder if there are interesting results that could be proven for the coefficients of the remaining columns in Figure \ref{shortchordtriangle}.  

Another interesting observation is that all matchings can be considered to be Fibonacci tableaux with no fixed points and all non-nesting matchings are both Fibonacci tableaux and $2 \times n$ tableaux that are counted by the Catalan numbers. Fibonacci tableaux with no fixed points correspond to permutations that are 312, 321 and 123 avoiding, thus giving a relation between matchings and pattern avoiding permutations. It would be interesting to see results involving various statistics on these objects translated to and/or from linear chord diagrams.

\acknowledgements
\label{sec:ack}
The authors would like to thank an anonymous referee whose helpful suggestions improved the paper.

\nocite{*}
\bibliographystyle{abbrvnat}
\bibliography{sample-dmtcs}
\label{sec:biblio}

\end{document}